\documentclass[a4paper, 11pt]{paper}
\usepackage{amssymb}
\usepackage{amsmath}
\usepackage{amsthm}

\usepackage{enumerate}
\usepackage[mathscr]{eucal}
\usepackage{graphics}
\usepackage[dvips]{graphicx}

\begin{document}



\let\goth\mathfrak


\newcommand*{\Land}{\;\land\;}
\newcommand\theoremname{Theorem}
\newcommand\lemmaname{Lemma}
\newcommand\corollaryname{Corollary}
\newcommand\propositionname{Proposition}
\newcommand\factname{Fact}
\newcommand\remarkname{Remark}
\newcommand\examplename{Example}

\newtheorem{thm}{\theoremname}[section]
\newtheorem{lem}[thm]{\lemmaname}
\newtheorem{cor}[thm]{\corollaryname}
\newtheorem{prop}[thm]{\propositionname}
\newtheorem{fact}[thm]{\factname}
\newtheorem{exmx}[thm]{\examplename}
\newenvironment{exm}{\begin{exmx}\normalfont}{\end{exmx}}

\newtheorem{rem}[thm]{\remarkname}

\def\myend{{}\hfill{\small$\bigcirc$}}

\newtheorem{reprx}[thm]{Representation}
\newenvironment{repr}{\begin{reprx}\normalfont}{\myend\end{reprx}}
\newtheorem{cnstrx}[thm]{Construction}
\newenvironment{constr}{\begin{cnstrx}\normalfont}{\myend\end{cnstrx}}
\def\classifname{Classification}
\newtheorem{classification}[thm]{\classifname}
\newenvironment{classif}{\begin{classification}\normalfont}{\myend\end{classification}}

\bibliographystyle{acm}
\newcommand{\minus}{{\bf -}}

\title{Multiplied configurations characterized by their closed substructures}
\author{Krzysztof Petelczyc \and Krzysztof Pra\.zmowski}
\pagestyle{myheadings}
\markboth{K. Petelczyc, K. Pra{\.z}mowski}{Closed substructures covering}

\def\LineOn(#1,#2){\overline{{#1},{#2}\rule{0em}{1,5ex}}}
\def\PointOf(#1,#2){{#1}\sqcap{#2}}
\def\lines{{\cal L}}
\def\collin{\sim}
\def\chain(#1){{#1}^{\ast}}
\def\inc{\mathrel{\rule{2pt}{0pt}\rule{1pt}{9pt}\rule{2pt}{0pt}}}
\def\dist{\mathrm{dist}}
\def\DifSpace(#1,#2){{\bf D}({#1},{#2})}
\def\CyclSpacex(#1,#2,#3){{#1}\circledast_{{#2}}{#3}}
\def\CyclSpace(#1,#2,#3){\circledast_{{#1}}({#2},{#3})}
\def\entier{\mathrm{E}}
\def\invers(#1){{#1}^{{-1}}}
\def\nwd(#1,#2){\mathrm{GCD}(#1,#2)}
\def\gras(#1,#2){{\bf G}_{#1}({#2})}
\def\otocz(#1,#2){{#1}_{({#2})}}
\def\alf(#1,#2){\mbox{${\strut}^{\alpha}\mkern-2mu{#1}_{({#2})}$}}
\def\bet(#1,#2){\mbox{${\strut}^{\beta}\mkern-2mu{#1}_{({#2})}$}}
\def\embfunc{\varepsilon}
\def\emb(#1,#2){\embfunc_{{#1}}({#2})}
\def\img{\mathrm{im}}
\def\ciach#1{\;\wr_{{#1}}\;}
\def\dod{\mathrel{\wr}}
\def\baero{\mathrel{\rho}}
\def\kor{\mathrel{\varkappa}}
\def\dod{\mathrel{\wr}}
\def\baer{\mathrel{\diamond}}
\def\Baer{\mathrel{\scriptstyle{\blacklozenge}}}
\def\corr{\mathrel{\vartriangleleft}}
\def\pls{partial linear space}
\def\id{\mathrm{id}}
\def\mappedby#1{%
{\rule{0pt}{2.5ex}}\mkern8mu{\longmapsto\mkern-25mu{\raise1.3ex\hbox{$#1$}}}
\mkern20mu{\rule{0pt}{4pt}}
}
\def\rank(#1){\mathrm{r}(#1)}
\def\qdots{\rule{0pt}{10pt}%
 {\raise-0.6ex\hbox{$\cdot$}}\mkern1.2mu{\raise0.1ex\hbox{$\cdot$}}%
 \mkern1.2mu{\raise0.8ex\hbox{$\cdot$}}
 }

\newcommand*{\struct}[1]{{\ensuremath{\langle #1 \rangle}}}
\newcommand*{\sub}{\raise.5ex\hbox{\ensuremath{\wp}}}
\def\Aut{{\text{\rm Aut}}}

  \def\pforall#1{(\forall{#1})}
  \def\pexists#1{(\exists{#1})}

\newcounter{sentence}
\def\thesentence{\roman{sentence}}
\def\labelsentence{\upshape(\thesentence)}

\newenvironment{sentences}{%
        \list{\labelsentence}
          {\usecounter{sentence}\def\makelabel##1{\hss\llap{##1}}
            \topsep3pt\leftmargin0pt\itemindent40pt\labelsep8pt}%
  }{%
    \endlist}

\maketitle

\begin{abstract}

We propose some new method of constructing configurations, which consists 
in consecutive inscribing copies of one underlying configuration.
A uniform characterization of the obtained class and the one introduced in
\cite{corset}, which makes use of some covering by family of closed 
substructures, is given.
\end{abstract}

\section{Introduction}

In a series of papers (cf. \cite{petel}, \cite{corset})
we have defined (and studied) two operations of "multiplying" \pls s;
the first operation can be applied to any \pls \ determined by a
quasi difference set (with this set ditinguished),
and the second one to any self dual structure
(with its correlation distinguished).
In both cases the resulting \pls \ $\goth N$
can be covered by some family $\cal B$ of its closed substructures,
each one isomorphic to the structure $\goth M$ which was multiplied.
In many examples this family $\cal B$ remains invariant under automorphisms of
$\goth N$;
if so the automorphisms of $\goth N$ can be easily determined.
\newline
There are some other operations which have this property. In section
\ref{sec:multdual} 
we give an example of such  a simple and natural operation, which does not
coincide
with the operations introduced in \cite{corset}.

\par
It turns out that some abstract, combinatorial properties
of the covering $\cal B$ are sufficient to characterize the geometry of $\goth N$.
Roughly speaking, two observations are crucial.
Let ${\cal B} = \big{\{} B_i \colon i \in I  \big{\}}$.
\begin{itemize}\def\labelitemi{\strut}\itemsep-2pt
\item
  If $B_i = (B'_i,B''_i)\in{\cal B}$ then on every line $L \in B''_i$
  there is exactly one point $L^\infty$ of $\goth N$ not in $B'_i$.
  All these "added" points are the points of some other
  closed substructure $B_j = (B'_j,B''_j)\in{\cal B}$.
\item
  This leads to a map $\baero \colon i \longmapsto j$, and maps
  $B''_i \ni L \longmapsto L^\infty \in B'_{\baero(i)}$.
  The first one determines on $I$ the structure of a cyclic group.
  The second are, in known examples, correlations.
\end{itemize}
In the paper we propose a system of conditions which express these properties in
a more elementary language and we prove the representation theorem.
Most of the previously investigated multiplied \pls s
satisfy this condition system, but the class of models satyfsying our conditions is even much wider.


\section{Multyplying by dualisation}     \label{sec:multdual}

First, we recall that an incidence structure $\goth M= \struct{S,\lines,\inc}$
with $\inc \; \subset S \times \lines$, is a {\em partial linear space} (PLS)
provided that its every line is on at least two points, every point is on at least two lines
and the following uniqueness condition is satisfied
$$\text{ if } a,b\in S,\; k,m\in\lines,\text{ and } a,b\inc k,m,
  \text{ then } a=b \text{ or } k=m.$$
Let us adopt a convention that in $\goth M$ sets $S$ and $\lines$ are disjoint. 
Now we introduce some notations we use further.
We write $k=\LineOn({a},{b})$ if $S \ni a,b \inc k\in \lines $ and $a\neq b$;
similarly $k\sqcap m=a$ means that $a\inc k,m$ and $k\neq m$.
The phrase $k\sqcap m=\emptyset$ is used when there is no such $a$
that $a\inc k,m$. If we write $a\sim b$ it means that there exists
in ${\goth M}$ a line $k$ such that $a,b\inc k$.
The degree of a point $p$ in ${\goth M}$ we denote by
$r_p:=|\{l\in\lines : p\inc l\}|$, and dually the size of a line
$r_l:=|\{p\in S: p\inc l\}|$.
We call the structure $\goth M$ {\em Shultenian} (or $\Gamma$-{\em space}) (cf. \cite{Cohen})
if for every triangle $(a,b,c)$
of $\goth M$ and for every point $d$ on the line $\LineOn(a,b)$ holds $c\sim d$.     

Let us remind
the construction of multyplying partial linear spaces using their correlations, 
considered in \cite{corset}.
\par
Let ${\goth M}_0=\struct{\mathrm{S}_0,\mathrm{L}_0,\inc_0}$ be a partial linear space 
with a correlation $\varkappa_0$ and let $k>2$ be an integer.
We define
${\goth M} = \CyclSpacex(k,\varkappa_0,{\goth M_0})$ as follows.
Let $M = C_k\times \mathrm{S}_0$, $\lines = C_k\times \mathrm{L}_0$.
We apply the following convention:
\begin{itemize}\def\labelitemi{--}\itemsep-2pt
\item  $(i,a)$ is a point, where $i\in C_k$ and $a\in S_0$;
\item  $[i,l]$ is a line,
       where $i\in C_k$ and $l\in\mathrm{L}_0$.
\end{itemize}
Then, the relation $\inc$ of incidence of
${\goth M}$ is characterized by the condition
\begin{equation}\label{wz:corelinc}
  (i,a)\inc[j,l] \text{ iff, either } i=j \text{ and } a\inc_0 l,\text{ or }
  i=j+1 \text{ and } a=\varkappa_0(l).
\end{equation}
and we set  ${\goth M}=\struct{M,{\lines},\inc}$.
The structure
  $$\CyclSpacex(k,\varkappa_0,{\goth M}_0) := \struct{M,{\lines},\inc}$$
will be referred to
as {\em a correlative multiplying of} ${\goth M}_0$.

\par
In this section
we adopt a dualisation (instead of correlation) as a convenient tool to multiply
partial linear spaces. Let ${\goth M}_0=\struct{\mathrm{S}_0,\mathrm{L}_0,\inc_0}$ be a
partial linear space, and
$\goth M_0^{\circ}$ be the partial linear space dual to ${\goth M}_0$.
We build a new structure $\goth M=\CyclSpacex(G,\circ,{\goth M}_0)$, where $G$ is a
cyclic group of an even (greater than $2$) or infinite rank, as follows.
Let $i\in G$, we put

$$M_i=\left\{ \begin{array}{ll}
               \{i\}\times \mathrm{S}_0 & \text{ for even } i \\
               \{i\}\times \mathrm{L}_0 & \text{ for odd } i,
              \end{array} \right. $$

$${\lines}_i=\left\{ \begin{array}{ll}
                      \{i\}\times \mathrm{L}_0 & \text{ for even } i \\
                      \{i\}\times \mathrm{S}_0 & \text{ for odd } i.
                     \end{array} \right. $$
Then we set
$M := \bigcup_{i\in G} M_i$ and
$\lines := \bigcup_{i\in G} \lines_i$.
According to the terminology, where $(i,a)$ is a point and $[i,b]$ is a line
of $\goth M$, we introduce the relation $\inc \; \subset M \times \lines$
as follows:
\begin{equation}\label{wz:relinc}
  (i,a)\inc[j,b] \text{ iff, either } i=j \text{ and } a\inc_0 b\;
  ( \text{or } b \inc_0 a),
  \text{ or }
  i=j+1 \text{ and } a=b.
\end{equation}
Finally, we put ${\goth M}=\struct{M,\lines,\inc}$.
\newline
Obviously, the structure $\CyclSpacex(G,\circ,{\goth M}_0)$ is a PLS.
Another immediate observation, based on the definition, is that
$|M|= |\lines|= \frac{|G|}{2}\cdot |\mathrm{S}_0|+\frac{|G|}{2}\cdot|\mathrm{L}_0|$.
The following is evident.
\begin{lem}\label{lem:parameters}
Assume that ${\goth M}_0$ is a PLS with constant point degree $\kappa$ 
and line size $\rho$. Let $a$ be a point or a line of ${\goth M}_0$, $i\in G$
and $\goth M=\CyclSpacex(G,\circ,{\goth M}_0)$.\\
$\begin{array}{cccc}
\begin{array}{l}
\text {If } (i,a) \text{ is a point of } \goth M \\
\text{then the degree of } (i,a) \text{ is} \\
\begin{array}{ccc}
\kappa+1 & \text{ if } & i \text{ is even}\\
\rho +1 & \text{ if } & i \text{ is odd}.
\end{array} 
\end{array}
& & &
\begin{array}{l}
\text {If } (i,a) \text{ is a line of } \goth M \\ 
\text{then the size of } (i,a) \text{ is} \\
\begin{array}{ccc}
\rho+1 & \text{ if } & i \text{ is even}\\
 \kappa+1 & \text{ if } & i \text{ is odd}.
\end{array}
\end{array} 
\end{array}$\\
Consequently, in $\goth M$ there are
$\frac{|G|}{2}\cdot |\mathrm{S}_0|$ points and the same number of lines with degree and size 
(respectively)
$\kappa+1$, and $\frac{|G|}{2}\cdot|\mathrm{L}_0|$ points and lines with
degree and size $\rho+1$. 
\end{lem}
%
The construction discussed in the paper has some advantages, when we compare it with
the one considered in \cite{corset}. Firstly, it can be applied to any partial linear space, not
necessarily admitting any correlation.
Note, however, that consequently, we arrive to not necessarily (regular) configurations.
The obtained structures are more artificial. 
\begin{exm}\label{exm:segment}
Let ${\goth M}_0$ be an incidence structure consisting of one line $c$ and two points $a,b$ on $c$,
i.e. ${\goth M}_0=\struct{\{a,b\},\{c\},\{a,b\}\times\{c\}}$. Such structure we call a segment.
Note, that point degree and line size in a segment are distinct.
Let $G=C_6$ and consider $\goth M=\CyclSpacex(G,\circ,{\goth M}_0)$ (see Figure \ref{fig:multex}).
 In $\goth M$ we observe two classes of points and of lines.
 One of them, derived from ${\goth M}_0$, contains points with degree $2$ and lines with
 size $3$. The second one, derived from $\goth M_0^{\circ}$, contains points with degree $3$ and lines with
 size $2$. As we see, the obtained structure is not a regular configuration.
\end{exm}
\begin{figure}[!h]
    \begin{center}
    \includegraphics[scale=0.6]{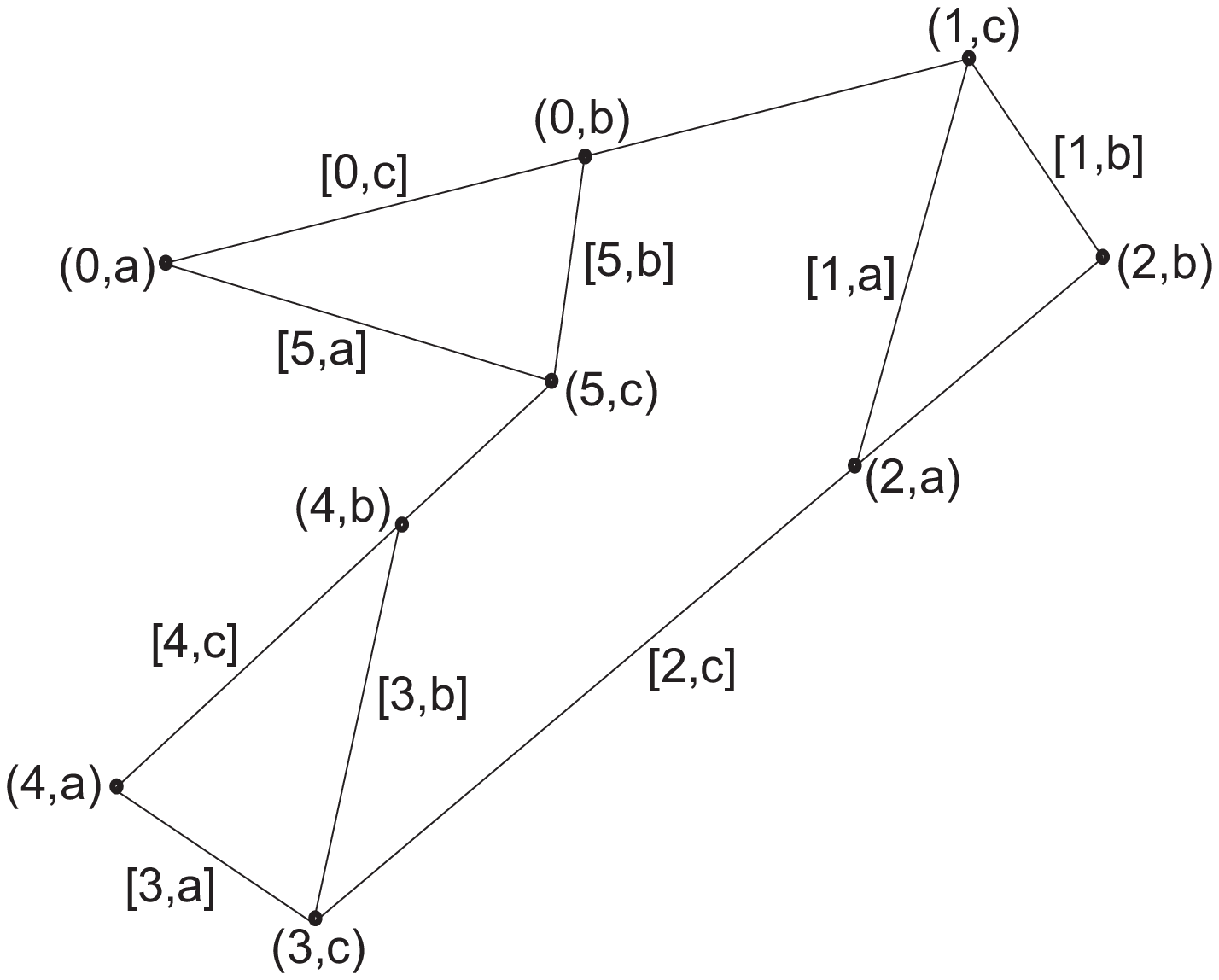}
    \end{center}
\caption{The configuration $\CyclSpacex(C_6,\circ,{\goth M}_0)$, where
${\goth M}_0$ is a segment (cf. Example \ref{exm:segment})
}
\label{fig:multex}
\end{figure}
\begin{prop}
The map $\kor=(\kor',\kor'')$,
$\kor':M \longrightarrow \lines$, $\kor'':\lines \longrightarrow  M$ defined by the formula
$$\kor'((i,x))=[1-i,x],\;\;\; \kor''([i,y])=(1-i,y)$$
is an involutive correlation of
$\goth M=\CyclSpacex(G,\circ,{\goth M}_0)=\struct{M,\lines,\inc}$.
Consequently, $\goth M $  is self dual.
\end{prop}
\begin{proof}
The map $\kor$ is a bijection of the structure $\goth M$, which transforms the
set of points onto the set of lines and dually. Consider $(i,a)$ and
$[j,b]$ such that $(i,a)\inc[j,b]$. The images of them under the map $\kor$ are
$[1-i,a]$ and $(1-j,b)$, respectively.

It is seen, that for $i=j$ the map $\kor$ preserves the relation $\inc$ in 
$\goth M$.
In the case when $i=j+1$ we get $[1-i,a]=[-j,a]$ and
$\kor''([j,b])\inc \kor'((i,a))$
since in that case $a=b$.
\end{proof}
\begin{prop}
Let $\goth M_0^{\circ}$ be dual to a partial linear space $\goth M_0$.
Then
$\CyclSpacex(G,\circ,{\goth M}_0)\cong \CyclSpacex(G,\circ,{\goth M}_0^{\circ})$.
\end{prop}
\begin{proof}
Note that the map $\delta$ of the form
$\delta((i,a))=(i+1,a)$, where $(i,a)\in G\times \mathrm{S}_0$ or 
$(i,a)\in G\times \mathrm{L}_0$ transforms
$\CyclSpacex(G,\circ,{\goth M}_0)$ onto 
$\CyclSpacex(G,\circ,{\goth M}_0^{\circ})$, and it is a required isomorphism.
\end{proof}
\begin{prop}
Let $\goth M=\CyclSpacex(C_k,\circ,{\goth M}_0)$, where $k$ is even.
If ${\goth M}_0$ is a self dual structure with an involutive correlation
$\kor_0$ then
$\goth M\cong \CyclSpacex(k,\kor_0,{\goth M}_0)$.
\end{prop}
\begin{proof}
Let us consider the following map
$\delta: \goth M \longrightarrow \CyclSpacex(k,\kor_0,{\goth M}_0)$:
$$\delta((i,x))=\left\{ \begin{array}{ll}
               (i,x) & \text{ for } i=2t\\
               (i,\kor_0(x)) & \text{ for } i=2t+1,
              \end{array} \right.$$
where $(i,x)\in C_k\times(\mathrm{S}_0\cup \mathrm{L}_0)$.
It is evident that $\delta$ is a bijection and it preserves
the relation of incidence.
Hence, this map is a required isomorphism.
\end{proof}


\subsection{Covering by closed substructures}

Now, we investigate some properties  of $\goth M=\CyclSpacex(G,\circ,{\goth M}_0)$
involving closed substructures. We are interested in similar results
to that obtained in \cite{corset}, but in our more general settings.
A substructure $\struct{B',B''}$ of
$\struct{M,\lines,\inc}$ is 
a {\em closed substructure} if it satisfies the following two conditions:
\begin{enumerate}[1]
\item
$\pforall{a_1,a_2\in B'}\pforall{l\in\lines}
[a_1,a_2\inc l \land a_1\neq a_2 \implies l\in B'']$,
\item
$\pforall{l_1,l_2\in B''}\pforall{a\in M}
[a\inc l_1,l_2 \land l_1\neq l_2 \implies a\in B']$.
\end{enumerate}
\begin{fact}\label{fct:imgBaer}
  Let ${\goth M}_0=\struct{\mathrm{S}_0,\mathrm{L}_0,\inc_0}$ be a partial linear space
  and $i \in G$.
  Consider 
  ${\goth M}=\CyclSpacex(G,\circ,{\goth M}_0)=\struct{M,\lines,\inc}$ and the map
  $\embfunc_i = (\embfunc'_i,\embfunc''_i)$ defined by

  $$
    \embfunc'_i\colon \left\{ \begin{array}{ll}
               a \mapsto (i,a) \in M        & \text{ for } i=2t\\
               a \mapsto [i,a] \in \lines   & \text{ for } i=2t+1
              \end{array} \right. $$

  $$\embfunc''_i \colon \left\{ \begin{array}{ll}
            l \mapsto [i,l] \in \lines & \text{ for } i=2t\\
            l \mapsto (i,l) \in M     & \text{ for } i=2t+1
             \end{array} \right.$$
for $a\in \mathrm{S}_0$, $l\in \mathrm{L}_0$.
  The image $\img(\embfunc_i)$
  of ${\goth M}_0$ under $\embfunc_i$ is a closed substructure of
  $\goth M$ for every $i \in G$. It is isomorphic to ${\goth M}_0$ or to
  ${\goth M}^{\circ}_0$, as $i$ is even or odd, respectively.
\end{fact}
Our goal is to characterize closed substructures of $\goth M=\CyclSpacex(G,\circ,{\goth M}_0)$
using terms of the geometry of $\goth M$.
 In \cite{corset} we propose certain external definitions:
\begin{equation}\label{def:baerpoints}
 \text{ if } p=(i,p'),q=(j,q') \in M \text{, then} \quad
 p\baer q \text{ iff } i=j\land
 \pexists{l\in \lines}\bigl[p,q\inc l\bigr],
\end{equation}
\begin{equation}\label{def:extrapoint}
    \text{ if } p=(i,p')\in M, l=[j,l'] \in \lines
    \text{, then } \quad
    p \dod l \text{ iff } p\inc l\land i=j+1.
\end{equation}
Consider above definitions in $\goth M$. The formula
\begin{equation}\label{def:baerpoints1}
  \pforall{p,q \in M}\Bigl[p\baer q \iff 
  \pexists{l\in \lines} 
    \bigl[p,q\inc l \;\land\; \neg(p\dod l) \;\land\;
    \neg(q\dod l)\bigr]\Bigr].
\end{equation}
express the relation $\baer$ in the language of $({\goth M},\dod)$ (cf. \cite{corset}).
 Assume that $\goth M_0$ is connected.
Let $\Baer\;\subseteq M\times M$ be the transitive closure of the relation $\baer$.
Then, the closed substructure $\img(\embfunc_i)$ of $\goth M$ is the
equivalence class of a point $(i,a)$ of ${\goth M}$
under the relation $\Baer$.
\par
The covering, mentioned in \ref{fct:imgBaer}, can be easily recovered
if the lines size and the points degree of the underlying structure ${\goth M}_0$ are
distinct.
\newline
We write $\rank(x)$ for degree of $x$, if $x$ is a point, or
for size of $x$, if $x$ is a line.
Then, straightforward from \ref{lem:parameters} we get:
\begin{cor}\label{pr:extrapoint0}
Let ${\goth M}_0$ be a finite configuration with point degree $\kappa$ and
line size $\rho\neq \kappa$.
Then for each point $p$ and each line $l$ of $\CyclSpacex(G,\circ,{\goth M}_0)$ we have
$$p\dod l \text{ iff } p\inc l\land \rank(p)=\rank(l).$$
\end{cor}
However, if the assumptions of \ref{pr:extrapoint0} are not valid, in particular
when ${\goth M}_0$ is self dual, we must use more complex methods to
restore the covering.
\newline
For $p\in M$, $l\in \lines$ we define the following relation:
\begin{multline}\label{def:addpoint}
  p \dod_1 l \iff  p\inc l \land \pexists{q\in M}
  \Bigl[q\neq p \land
  q\inc l \land \pforall{k,m,n\in \lines} \bigl[ l \neq k,m,n \land
  m\neq n \land
  \\
   q\inc k \land  p \inc m,n   
    \land k\cap m\neq \emptyset
     \implies k\cap n = \emptyset\bigr]\Bigr].
\end{multline}
\begin{prop}\label{pr:extrapoint1}
  Let $\CyclSpacex(G,\circ,{\goth M}_0)=\struct{M,\lines,\inc}$
  and ${\goth M}_0$ be Shultenian.
  Moreover, assume that ${\goth M}_0$ has no $2$-element lines, and any two
  collinear points of ${\goth M}_0$ can be completed to a triangle in ${\goth M}_0$
  (equivalently: every line of ${\goth M}_0$ is contained in a plane), and
  every two intersecting lines from ${\goth M}_0$ determine a plane.
  Then $p \dod_1 l$ iff $p\dod l$
  for every $p \in M$ and $l \in \lines$.
\end{prop}
\begin{proof}
  Let $\goth M=\CyclSpacex(G,\circ,{\goth M}_0)=\struct{M,\lines,\inc}$, where
  ${\goth M}_0=\struct{\mathrm{S}_0,\mathrm{L}_0,\inc_0}$,
  and $p \in M_i\subset M$, $l \in \lines_j\subset \lines$.
\par
  Let $i$ be an even number from a cyclic group $G=C_k$ ($k>2$) with even $k$,
  or from $G={\mathbb Z}$.
 There are $p'\in \mathrm{S}_0$, $l'\in \mathrm{L}_0$ such that
 $p=(i,p')$  and $l=[j,l']$.
 Let us assume that $p\dod l$,
 in other words $p\inc l$ and $i=j+1$.
 Take a point $q\neq p$ and a line $k\neq l$ such that $q\inc l,k$.
 Next, consider a line $m$ such that $m\sqcap l=p$, $m\sqcap k=r$.
 Observe $\otocz({\goth M},p)$ (see fig.\;\ref{fig:multneighb}).
 Note that neighbourhoods of a point
 in the structure $\goth M$
 and in $\CyclSpacex(k,\varkappa,{\goth M}_0)$, considered in \cite{corset}
 (where $\varkappa$ is the involutory correlation of ${\goth M}_0$) are isomorphic
 both for $k=3$ and for an arbitrary $k>3$.
 Therefore, the line $m$ is the unique one in $\otocz({\goth M},p)$,
 which crosses any line passing through $q$, it follows from
 \cite[Lemma 2.4]{corset}.
  Consequently, $p \dod_1 l$ holds.
 \par
 Now, assume that $p \dod l$ does not hold. For $p$ not laying on $l$
 we get $\neg(p \dod_1 l)$ immediately.
 Let $\neg(p \dod l)$ and  $p\inc l$ hold, so $i=j$.
 Consider arbitrary point $q = (i,q') \neq p$ such that $q\inc l$.
 From assumptions there exists a triangle $(q',p',r')$ in ${\goth M}_0$ and
 a point $s'\neq p',r'$ on $\LineOn(q',r')$. Set
   $k = [i,\LineOn(q',r')]$, $m = [i,\LineOn(p',r')]$, and
 $n = [i,\LineOn(p',s')]$, which existence follows by the Shult axiom.
 Next, let $q = (i+1,l')\inc l$ and $k = [i+1,p']$, so $q \inc k$.
 Take any two lines $m,n\neq l$ of the form $[i,x']$ passing through $p$;
 then $k$ crosses them both (see details in \cite{corset}).
 This proves that $p \dod_1 l$ does not hold.
\par
  If $i\in G$ is odd then for $p=(i,p')$, $l=[j,l']$ we have $p'\in \mathrm{L}_0$,
  $l'\in \mathrm{S}_0$.
  If we assume that $p\dod l$ then 
 directly from analysing
 in $\CyclSpacex(G,\circ,{\goth M}_0)$
 the neighbourhood of the point $p$ we obtain that $p \dod_1 l$. 
\par 
 In order to close our proof we consider the case with $i$ odd and
 $\neg(p \dod l)$.  
 Let $q'\in \mathrm{L}_0$ be a line such that $q'$ intersects $p'$.
 Every two intersecting lines
 determine a plane in ${\goth M}_0$.  
 Hence, we get the existance of a line $r'\in \mathrm{L}_0$ such that $q',p',r'$
 are sides of one triangle. 
 Now, the claim is provided by arguments analogous to those used for
 $i$ even.
\end{proof}
\begin{figure}[!h]
    \begin{center}
    \includegraphics[scale=0.6]{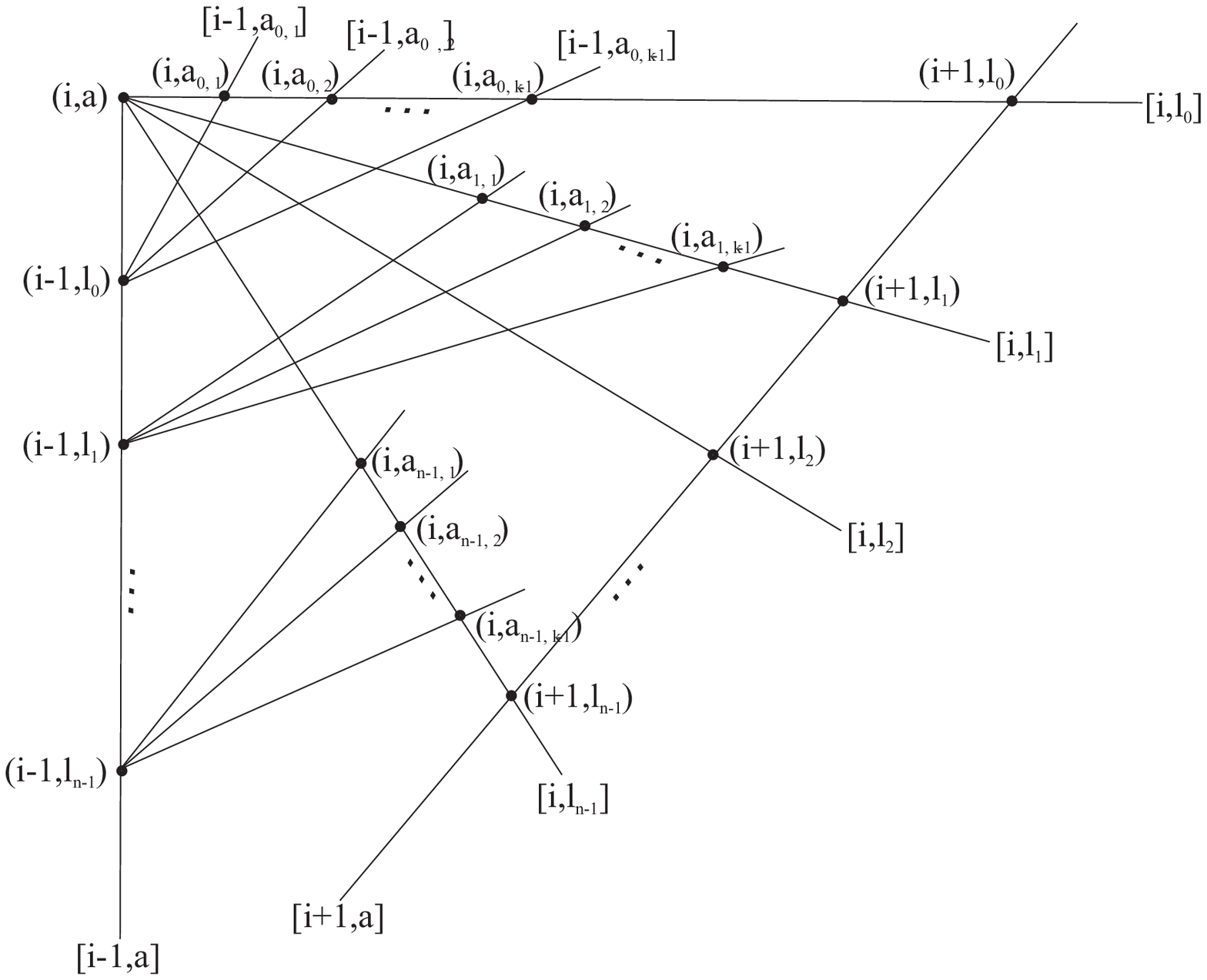}
    \end{center}
\caption{$\otocz(\goth M,{(i,a)})$ -- the neighbourhood of a point $(i,a)$ in
$\goth M=\CyclSpacex(G,\circ,{\goth M}_0)$}
\label{fig:multneighb}
\end{figure}
Let us introduce another two relations $\corr,\dod_2  \;\; \subseteq M \times \lines$ given
by the formulas:
\begin{equation}\label{def:pointcorline}
  a \corr l \quad\iff\quad \neg(a\inc l)\Land (\exists! \; m) \;
  [a \inc m \land l\sqcap m = \emptyset];
\end{equation}
\begin{equation}\label{def:extrapoint2}
 p\dod_2 l \iff \pexists{a\in M}[l\inc p\nsim a\corr l].
\end{equation}
\begin{prop}\label{pr:extrapoint2}
 Let $\CyclSpacex(G,\circ,{\goth M}_0)=\struct{M,\lines,\inc}$
 and ${\goth M}_0=\struct{\mathrm{S}_0,\mathrm{L}_0,\inc_0}$ be a partial linear space
 with constant size $>2$ of each line. Assume, that for every pair 
 $(a,l)\in \mathrm{S}_0\times \mathrm{L}_0$ such
 that $a$ is outside $l$ there is a line through $a$, which misses $l$ 
 and there is a point on $l$, which is not collinear with $a$.
 Then $p \dod_2 l$ iff $p\dod l$ for every $p \in M$ and $l \in \lines$
\end{prop}
\begin{proof}
Let $\goth M=\CyclSpacex(G,\circ,{\goth M}_0)=\struct{M,\lines,\inc}$.
\par
Let us take a point $p=(i,p')\in M$ and a line $l=[j,l']\in \lines$ such that
$p\dod l$. Then $p\inc l$ and $i=j+1$. Let us consider
the neighbourhood of a point $a$ such that $a\corr l$. Since $a\corr l$, then
$\rank(a)-1$ is the size of $l$ in $\otocz(\goth M,a)$.
From assumptions, for every line $m=[j-1,m']\in \lines$, which does not pass through $a$,
there is a line from $\img(\embfunc''_{j-1})$ through $a$, which misses $m$. 
Thus, we have at least two lines through $a$, which do not cross $m$ in $\otocz(\goth M,a)$.
In $\otocz(\goth M,a)$ size of every line from $\img(\embfunc''_{j-2})$, 
which does not pass through $a$, equals $2$.
Hence, in $\otocz(\goth M,a)$ size of every line distinct from $l$,
which does not pass through $a$ is less than
$\rank(a)-1$.
Therefore, the line $l$ is uniquely determined by its relation to $a$.
The point $p$ is also the unique one on the line $l$, which is not
collinear with $a$. Consequently, we conclude that $p\dod_2 l$.
\par
The converse implication can be proved by applying the converse reasoning.
\end{proof}
Propositions \ref{pr:extrapoint0}, \ref{pr:extrapoint1} and \ref{pr:extrapoint2} together with \ref{fct:imgBaer}
yield the following:
\begin{cor}\label{pr:corneb:inf}
  Let ${\goth M}_0$ be connected.
  Under assumptions of each of \ref{pr:extrapoint0}, \ref{pr:extrapoint1} and \ref{pr:extrapoint2}
  the covering of $\CyclSpacex(G,\circ,{\goth M}_0)$ by the family of closed 
  substructures $\{\img(\embfunc_i):i \in G \}$ is definable in 
  $\CyclSpacex(G,\circ,{\goth M}_0)$. 
 Consequently it is preserved by all 
  automorphisms of $\CyclSpacex(G,\circ,{\goth M}_0)$.  
\end{cor}
There are some significant examples of structures  $\CyclSpacex(G,\circ,{\goth M}_0)$, of which automorphisms
preserve their closed substructures.
Proposition \ref{pr:extrapoint0} provides a wide class of such examples. For instance we
can adopt any finite affine, or slit space (cf. \cite{karzel1}, \cite{karzel2})
as a structure ${\goth M}_0$ we start from.
Every projective plane satisfies assumptions of \ref{pr:extrapoint1}. Thus, the family of $\CyclSpacex(G,\circ,{\goth M}_0)$, where
${\goth M}_0$ is a projective plane, is the next class of examples.
\par
Let $X$ be an arbitrary set. We write $\sub_m(X)$ for the set of $m$-element subsets of $X$.
Let $m$ be an integer with $1\leq m < |X|$. Then $\gras(m,X)$ is the incidence 
structure 
$\gras(m,X) = \struct{\sub_m(X),\sub_{m+1}(X),\subset}$. We call this structure 
{\em combinatorial Grassmannian}.
It is worth to note that $\gras(2,5)=\goth D$ is the Desargues configuration.
\begin{fact}
If $m\neq n-1,n-2$, then the structure $\gras(m,X)$ satisfies assumptions of \ref{pr:extrapoint2}.
\end{fact}
\begin{proof}
Consider a point $a$ and a line $L$
  of $\gras(m,X)$, such that $a$ is outside $L$. If $a\cap L = \emptyset$ then we set
  $M:= a\cup \{x\}$ with $x\in L$, and $b:=L\setminus \{x\}$. 
  \newline
  Let $a\cap L \neq \emptyset$. From assumptions there exists $x$ such that
  $x\notin a\cup L$. Then we set $M:= a\cup \{x\}$.
  There exist two $y,y'\in L$ such that $y,y'\notin a$, since $|a\cap L|\leq m-1$.
  Take $b\subset L$ with $y,y'\in b$.
  \newline
  In both above cases $|M\cap L|<m$ and $|b\cup a|>m+1$. 
\end{proof}
That way we obtain the second class of examples -- structures
$\CyclSpacex(G,\circ,{\gras(m,X)})$.
\par
Let ${\goth F} = \struct{F,+,\cdot,0,1}$ be a finite field, ${\goth F} = GF(q)$.
Assume that $2\nmid q$.
We construct the {\em Havlicek-Tietze configuration} $HT(q)$ as follows:
\begin{itemize}\def\labelitemi{--}\itemsep-2pt
\item
  Its point universe is $X = F\times F$, with elements of $X$ written as $(a,b)$,
  $a,b\in F$.
\item
  Its blocks are pairs $[\alpha,\beta]$ with $\alpha,\beta\in F$, we write ${\cal G}$
  for the set of such blocks.
\item
  The incidence relation is defined by 
  $(a,b) \inc [\alpha,\beta]$ iff $a\cdot\alpha = b + \beta$.
\end{itemize}
Then $HT(q) := \struct{X,{\cal G},\inc}$.
Note that $HT(3)$ is the Pappus configuration.
In \cite{corset} we proved the following:
\begin{fact}
 Let ${\goth A}$ be an affine plane over $\goth F$   
  and $\cal D$ be the direction of a line $l$ of $\goth A$.
  Then 
  $HT(q)$ results from $\goth A$ by deleting the lines in $\cal D$.
  In particular, $HT(q)$ is a partial linear space with parallelism.
  Conversely, the family 
  $\{ c \colon  a = c \text{ or }  a = b \text{ or } 
   c \text{ is not collinear with }  a, b\}$, where $a$, $ b$
  are not collinear points of $HT(q)$, is the set $\cal D$.
  \end{fact}
  Now, it is seen that the assumptions of \ref{pr:extrapoint2} hold for $HT(q)$.
 Hence, another set of examples in question is
 $\CyclSpacex(G,\circ,{HT(q)})$.


\section{Synthetic characterization}

Let $B = \{B_i=\struct{B'_i,B''_i}: i\in I\}$ be a family
of connected closed substructures of a connected partial linear space
${\goth M}=\struct{M,\lines,\inc}$ such that
$\bigcup_{i\in I}B'_i=M$, $\bigcup_{i\in I}B''_i=\lines$.
Let us introduce the following conditions
\begin{enumerate}[(1)]
\item
$\pforall{i_1,i_2\in I} [B'_{i_1} \cap B'_{i_2}\neq \emptyset
\implies B'_{i_1}= B'_{i_2}]$ \label{ax:emporeq}
\item
$\pforall{i_1,i_2\in I} [B''_{i_1} \cap B''_{i_2}\neq \emptyset
\implies B''_{i_1}=B''_{i_2}]$ \label{ax:dualemporeq}
\item
$\pforall{d\in B'_i}\pexists{m} [d\inc m \land m\notin B''_i]$
\label{ax:onetoall}
\item
$\pforall{m\in B''_i}\pexists{d} [d\inc m \land d\notin B'_i]$
\label{ax:dualonetoall}
\item
$\pforall{i\in I}[\struct{M\setminus B'_i,\lines\setminus B''_i}
\text{ is a closed substructure of }\struct{M,\lines}]$ \label{ax:antybaer}
\item
$\pforall{m_1,m_2,m_3}\bigl[p\inc m_1,m_2,m_3\land d_1\inc m_1\land d_2\inc m_2
\land d_3\inc m_3\land d_1,d_2,d_3\notin B'_i
\implies \pexists{n}[d_1,d_2,d_3\inc n
\vee m_1\notin B''_i \vee m_2\notin B''_i
\vee m_3\notin B''_i]\bigr]$ \label{ax:cor3}
\item
$\pforall{d_1,d_2,d_3}\bigl[d_1,d_2,d_3\inc n\land d_1\inc m_1\land d_2\inc m_2
\land d_3\inc m_3\land m_1,m_2,m_3\notin B''_i
\implies \pexists{p}[p\inc m_1,m_2,m_3
\vee d_1\notin B'_i \vee d_2\notin B'_i
\vee d_3\notin B'_i]\bigr]$ \label{ax:dualcor3}
\end{enumerate}
Note, that conditions in the following pairs: \eqref{ax:emporeq} and \eqref{ax:dualemporeq}, 
\eqref{ax:onetoall} and \eqref{ax:dualonetoall},  \eqref{ax:cor3} and \eqref{ax:dualcor3},
 are mutually dual.
\par
\noindent
Now, let us define a new relation $\baero\;\; \subseteq I\times I$, namely:
\begin{equation}\label{def:przechod}
  j\baero i  \iff \pexists{l\in B''_j}\pexists{a\in B'_i}
  [a\inc l \land j\neq i].
\end{equation}
Then, as an immediate consequence of conditions \eqref{ax:onetoall} and
\eqref{ax:dualonetoall}, we obtain
\begin{cor}\label{cor:surjectivity}
For every $i\in I$ there exists $j\in I$  and
for every $j\in I$ there exists $i\in I$ such that $i\baero j$.
\end{cor}
\noindent
Let us pay some more attention to the properties of $B$, which follow
directly from the introduced conditions.
\begin{lem}\label{lem:cor2}
The following statement holds:
\begin{enumerate}[]
\item
$\pforall{m_1,m_2}\bigl[p\inc m_1,m_2\land d_1\inc m_1\land d_2\inc m_2
\land m_1\in B''_i \land m_2\in B''_i \land d_1,d_2 \notin B'_i
\implies \pexists{n}[d_1,d_2\inc n ]\bigr].$
\end{enumerate}
Furthermore, the dual version of this statement is also valid.
\end{lem}
\begin{proof}
It suffices to adopt $m_2=m_3$ and $d_2=d_3$ in conditions
\eqref{ax:cor3}, \eqref{ax:dualcor3}.
\end{proof}
\begin{lem}\label{lem:function}
Let $d_1\inc m_1$, $d_2\inc m_2$, $m_1\sqcap m_2\neq\emptyset$ and
$m_1, m_2 \in B''_j$.
If $d_1\in B'_i$ for some $i\neq j$ and $d_2\notin B'_j$, then
$d_2\in B'_i$.
Moreover, the dual version of this statement is true as well.
\end{lem}
\begin{proof}
Let us take some $I\ni i_1,i_2\neq i$. Assume that $d_1\in B'_{i_1}$,
$d_2\in B'_{i_2}$ , $i_1\neq i_2$,
and $d_1\inc m_1\in B''_j$, $d_2\inc m_2\in B''_j$ and $m_1\sqcap m_2\neq\emptyset$.
There exists a line $n$ which joins $d_1,d_2$, from \ref{lem:cor2}.
Then, condition \eqref{ax:antybaer} yields that $n\notin B''_j$.
Note, that $m_1,m_2\notin B''_{i_1},B''_{i_2}$. So, if moreover
$n\notin B''_{i_1}$ or $n\notin B''_{i_2}$ then
$d_1\notin B'_{i_1}$ or $d_2\notin B'_{i_2}$
follows from condition \eqref{ax:antybaer}.
Hence, $n\in B''_{i_1}$ and $n\in B''_{i_2}$ holds.
Together with condition \eqref{ax:dualemporeq} it implies that $i_1=i_2$,
which closes our proof.
\end{proof}
Note, that since each of close substructure of $\goth M$ is connected the
conclusion of \ref{lem:function} remain also valid for $m_1\sqcap m_2 =\emptyset$.
Then, using \ref{cor:surjectivity} and \ref{lem:function},  we obtain
\begin{cor}\label{cor:bijectivity}
The map $\baero$ is a bijection of the set $I$.
\end{cor}
\noindent
Directly from definitions we have
\begin{fact}\label{fact:connected}
If $a_j\in B'_{i_j}$, $j=1,2$ and $a_1\sim a_2$ then $i_1=i_2$, 
$i_1\baero i_2$ or $i_2\baero i_1$.
\end{fact}
\begin{prop}
For arbitrarily fixed $i_0\in I$ and the map $\baero$ defined in 
\eqref{def:przechod} the following holds:
$$\struct{\baero}[i_0]=\{\baero^s(i_0): s=0,\pm 1,\pm 2,\ldots\}=I.$$
Then, the map $\baero$ determines on $I$ the structure of a group
$C_k$ (with $k=|I|$) for a finite set $I$ or the structure of
$\mathbb{Z}$ otherwise. 
Moreover, $i\baero j$ iff $j=i+1$.
\end{prop}
\begin{proof}
The set $\struct{\baero}$ is a group of bijections, from \ref{cor:bijectivity}.
Let us fix some $i_0\in I$. Then, $\struct{\baero}[i_0]$ is an orbit of
the group $\struct{\baero}$. The equality 
$\struct{\baero}[i_0]=I$ is immediate from \ref{fact:connected} and
connectedness of $\goth M$.
\par
Obviously $\struct{\baero}$ is a cyclic group. Then, either 
$\struct{\baero}\cong C_k$ if $k$ is a finite rank of $\struct{\baero}$ or
$\struct{\baero}\cong\mathbb{Z}$ otherwise. We denote the rank of the group
$\struct{\baero}$ as $r_{\struct{\baero}}$.
For $s\in C_k$ (or $s\in \mathbb{Z}$) let us consider the map
$$s\mappedby{\gamma}\baero^s(i_0).$$
One can note that $\gamma$ induces the following map
$\struct{\baero}\ni\baero^s\mappedby \baero^s(i_0)\in I$.
\par
We claim that $\gamma$ is a bijection.
Let us assume that $\baero^{s_2}(i_0)=\baero^{s_1}(i_0)$. It yields 
$\baero^{s_2-s_1}(i_0)=i_0$ for some $s_1, s_2$ such that 
$r_{\struct{\baero}}>s_2>s_1$.
Therefore $\baero^{k}(i_0)=i_0$ holds for $k<r_{\struct{\baero}}$.
If 
$I_1:=\{\baero^j(i_0): j=0,\ldots,k-1\}$ then $\baero^s(I_1)=I_1$ holds
for any $s\in \mathbb{Z}$. 
Thus, $\struct{\baero}[i_0]\subset I_1$ and consequently $I=I_1$.
It implies that $\baero^k(i)=i$ for any $i\in I$. Finally, we obtain
that the rank of $\struct{s}$ is less than $k$, which contradicts 
our previous assumptions.
\par
Let $s'=s+1$ hold. 
From \ref{cor:bijectivity} 
$\baero(\baero^s(i_0))=\baero^{s+1}(i_0)$ and $i$, $\baero(i)$ are
$\rho$-related for any $i\in I$.
Hence, $\gamma(s)\baero \gamma(s')$ holds. This proves that $\gamma$ transfers
the structure of the group $C_k$ (or $\mathbb{Z}$) on $\struct{\baero}[i_0]$.
\end{proof}
\begin{lem}\label{lem:correlation}
Let $m_1,m_2,m_3\in B''_j$ and $d_1,d_2,d_3\in B'_i$ such that
$d_1\inc m_1$, $d_2\inc m_2$, $d_3\inc m_3$ for some $j,i\in I$, $j\neq i$.
Then, $m_1,m_2,m_3$ meet in a point distinct from $d_1,d_2,d_3$ iff
$d_1,d_2,d_3$ are on a line distinct from $m_1,m_2,m_3$.
\end{lem}
\begin{proof}
Let $p\neq d_1,d_2,d_3\in B'_i$ be the common point of the lines
$m_1,m_2,m_3\in B''_j$. The condition \eqref{ax:emporeq} yields
$d_1,d_2,d_3\notin B'_j$, since $j\neq i$ .
Then, directly from condition \eqref{ax:cor3} we get the existence of the line
$n\neq m_1,m_2,m_3$, which connects $d_1,d_2,d_3$.

Analogously, based on conditions \eqref{ax:dualemporeq} and \eqref{ax:dualcor3},
we prove the converse implication.
\end{proof}

Now, we define another relation $\dod\;\; \subseteq M \times \lines$ by the formula
\begin{equation}\label{def:extrapoint}
 a\dod m \text{ iff } a\inc m \text{ and } \neg \pexists{i\in I}
 [a\in B'_i \land m\in B''_i]
\end{equation}
\begin{lem}\label{lem:dodisfunction}
The relation $\dod\;\; \subseteq M \times \lines$ defined in \eqref{def:extrapoint}
is a bijection acting from $M$ onto $\lines$.
\end{lem}
\begin{proof}
From condition \eqref{ax:onetoall} we get that for every $a\in M$ there 
exists such $m\in \lines$ that $a\dod m$. Assume that
there exist $a\in M$ and $m_1,m_2\in \lines$ such that
$a\dod m_1$ and $a\dod m_2$. Then, from condition
\eqref{ax:antybaer}, $m_1=m_2$. Thus, relation $\dod$ is a function.
Conditions \eqref{ax:dualonetoall} and \eqref{ax:antybaer} yield 
that $\dod$ is a surjection and an injection, respectively.
\end{proof}
Next, let us introduce the following pair of transformations:
\begin{equation}\label{def:subcor}
\left\{ \begin{array}{ll}
               \kor'(a)=m \text{ iff } a\dod m\\
               \kor''(m)=a \text{ iff } a\dod m,
              \end{array} \right.
\end{equation}
which are in fact functions, what follows from \ref{lem:dodisfunction}.
Take $j,i\in I$ such that $j\baero i$. Then we have $\kor'(B'_j)= B''_i$
and $\kor''(B''_i)= B'_j$. To get a proper, full correlation
acting from the structure $B_j$ to $B_i$ let us note the following.
\begin{fact}\label{fact:corofsubstr}
Let $j\baero i$ for $j,i\in I$ and $\kor=(\kor',\kor'')$ be the function
introduced in \eqref{def:subcor}, and
$\phi'_j:=\kor'\restriction B'_j$, $\psi''_i:=\kor''\restriction B''_i$.
There exist maps $\phi''_j\colon B''_j\longrightarrow B'_i$,
$\psi'_i\colon B'_i\longrightarrow B''_j$, which
complete $\phi'_j$ and $\psi''_i$ to the correlations
$$\phi_j=(\phi'_j,\phi''_j),\;\;\; \psi_i= (\psi'_i,\psi''_i),$$
where
$\phi_j\colon B_j \longrightarrow B_i$, $\psi_i\colon B_i \longrightarrow B_j$,
and $\phi''_j=(\psi'_i)^{-1}$, $\psi''_i=(\phi'_j)^{-1}$ hold.
\end{fact}
\begin{proof}
The existence of maps $\phi''_j:B''_j \longrightarrow B'_i$ and
$\psi'_i:B'_i \longrightarrow B''_j$ is immediate by \ref{lem:correlation}.
If $j\baero i$ then the transformation $\psi'_i$ is the inverse of $\phi''_j$,
and $\phi'_j$ is the inverse of $\psi''_i$.
\end{proof}


\section{Representation theorems}

Now we present, in some sense, more general construction. 
We set  $I=C_k$ or $I=\mathbb{Z}$.
Let us consider a family of connected partial linear spaces
$(\goth M_i)_{i\in I}=\struct{M_i,\lines_i,\inc_i}$ and
a family $(\phi_i)_{i\in I}$,
where $\phi_i=(\phi'_i,\phi''_i)$ is a correlation such that
$\phi_i\colon {\goth M}_i\longrightarrow {\goth M}_{i+1}$. Then we put
$$ X_i := \{i\}\times M_i,  \;\;\; X = \bigcup_{i\in I} X_i,
\quad\quad
H_i := \{i\}\times \lines_i, \;\;\; H = \bigcup_{i\in I} H_i,$$
and we introduce the following relation $\inc$
\begin{equation}\label{wz:relinc1}
  (i,a)\inc[j,m] \text{ iff, either } i=j \text{ and } a\inc_i m,\text{ or }
  i=j+1\text{ and }a=\phi''_j(m).
\end{equation}
As a result we obtain the structure
\begin{equation}\label{def:glue}
  \CyclSpace(i\in I,\goth M_i,\phi_i):=\struct{X,H,\inc}.
\end{equation}
Clearly ${\goth B}_i=\struct{X_i,H_i}$ is a closed substructure of
$\goth M=\CyclSpace(i\in I,\goth M_i,\phi_i)$ isomorphic to $\goth M_i$, for
each $i\in I$, and $\goth M$ is a partial linear space.
As an another straightforward consequence of the definition we get
\begin{fact}\label{satofaxioms}
The family $\{{\goth B}_i: i\in I\}$ of closed
substructures of the structure $\goth M$ is a covering of
$\goth M$.
This family satisfies all conditions from
\eqref{ax:emporeq} to \eqref{ax:dualcor3}.
\end{fact}
%
%
\begin{prop}
Let $\goth M_i=\struct{M_i,\lines_i,\inc_i}$ be a connected partial linear
space for each of $i\in I$ and $(\xi_i)_{i\in I}=(\xi'_i,\xi''_i)$ be a
correlation of $\goth M_i$ onto $\goth M_{-i}$. Next, let
$\goth M=\CyclSpace(i\in I,\goth M_i,\phi_i)=\struct{X,H,\inc}$  be
the structure defined by \eqref{def:glue}.  Assume that
$$\phi''_{-i} \xi'_i \phi''_{i-1}= \xi''_i.$$
Then, the structure $\goth M$
is self dual and the map $\kor=(\kor',\kor'')$,
$\kor':X \longrightarrow H$, $\kor'':H \longrightarrow X$, where
$$\kor'((i,x))=[-i,\xi'_i(x)],\;\;\; \kor''([i,y])=(-i,\xi''_i(y))$$
is a correlation of $\goth M$.
\end{prop}
\begin{proof}
The map $\kor$ is a bijection of the structure $\goth M$, which transforms the
set of points onto the set of lines and dually. Let us take such $(i,a)$ and
$[j,m]$, that $(i,a)\inc[j,m]$. Their images are $\kor'((i,a))=[-i,\xi'_i(a)]:=[a']$
and $\kor''([j,m])=(-j,\xi''_i(m)):=(m')$.

Let us assume that $i=j$ and $a\inc_i m$. Then, $(m')\inc [a']$ as the map
$\xi_i$ is a correlation transforming ${\goth M}_i$ onto ${\goth M}_{-i}$.

Now, assume $i=j+1$ and $a=\phi''_j(m)$. So, $(m')\inc [a']$ if
$\phi''_{-i}(\xi'_i(a))=\xi''_i(m)$. And next, we get
$\phi''_{-i}(\xi'_i(\phi''_{i-1}(m)))=\xi''_i(m)$, which is the required condition.
\end{proof}
%
%
%
%
%
%
%
\begin{thm}
Let ${\goth M}=\struct{M,\lines,\inc}$ be a
connected partial linear space covered by
a family of closed substructures $\{B_i=\struct{B'_i,B''_i}: i\in I\}$ satisfying
all conditions from \eqref{ax:emporeq} to \eqref{ax:dualcor3}.
Let $\baero$ be relation introduced in \eqref{def:przechod}, 
and $\phi_i= (\phi'_i,\phi''_i)$ be the
correlation, defined in \ref{fact:corofsubstr}, mapping $B_i$ onto $B_{\rho(i)}$. 
Then $\goth M\cong\CyclSpace(i\in I,B_i,\phi_{i})$.
\end{thm}
\begin{proof}
Let us define a map 
$\delta\colon \goth M \longrightarrow \CyclSpace(i\in I,B_i,\phi_{i})$ 
by the formula:
$$ \begin{array}{ll}
               \delta'\colon a\mapsto (i,a)
               & \text{ for } a\in B'_i\\
               \delta''\colon l\mapsto [i,l]
               & \text{ for } l\in B''_i
              \end{array}  \textrm{,\; and we put } \delta=(\delta',\delta'').$$
Since the family $\{B_i=\struct{B'_i,B''_i}: i\in I\}$ covers $\goth M$, every $a\in M$ and
every $l\in\lines$ has an image under $\delta$. From conditions 
\eqref{ax:emporeq} and \eqref{ax:dualemporeq} this image is uniquely determined.
Then the map $\delta$ is a function. Directly from the definition it is a surjection and
an injection.
\newline
Assume that $a\inc l$ in $\goth M$. Let $a\in B'_i$, $l\in B''_i$.
Then $\delta'(a)=(i,a)$, $\delta''(l)=[i,l]$, and $\delta'(a)\inc \delta''(l)$
follows by \eqref{wz:relinc1}.
\newline
Let $a\in B'_i$, $l\in B''_j$ and $i\neq j$. From \eqref{def:przechod} we get $j\baero i$.
Thus, \ref{fact:connected} yields $i=j+1$, and $\delta'(a)=(i,a)\inc [j,l]=\delta''(l)$ follows by
\eqref{wz:relinc1}.
\end{proof}
Note, that a family
$\{\delta(B_i)\colon i\in I\}$
is a covering of $\CyclSpace(i\in I,B_i,\phi_{i})$ by a family of its closed substructures. 
For $a\in B'_i$, $l\in B''_i$ let us consider another map
$\phi^{\star}_i= ({\phi^{\star}_i}',{\phi^{\star}_i}'')$ such that
$${\phi^{\star}_i}'\colon (i,a)\mapsto [i+1,\phi'_i(a)]
\text{, \;\;and\;\; } {\phi^{\star}_i}''\colon [i,l]\mapsto (i+1,\phi''_i(l)).$$
The transformation $\phi^{\star}_i$ is a
correlation, induced by $\phi_i$,
mapping $\delta(B_i)$ onto $\delta(B_{i+1})$.
%
\begin{prop}
Let $\CyclSpace(i\in I,\goth M_i,\phi_i)$ be the structure defined by
\eqref{def:glue} and ${\goth M}_0=\struct{\mathrm{S}_0,\mathrm{L}_0,\inc_0}$ be a PLS.
\begin{sentences}
\item
If $I=\mathbb{Z}$ then
$\CyclSpace(i\in I,\goth M_i,\phi_i)\cong \CyclSpacex(\mathbb{Z},\circ,{\goth M}_0)$.
\item
If $I=C_k$ with $k$ even and $\phi_{k-1}\ldots \phi_1\phi_0=\id$ then
$\CyclSpace(i\in I,\goth M_i,\phi_i)\cong \CyclSpacex(C_k,\circ,{\goth M}_0)$.
\end{sentences}
\end{prop}
\begin{proof}
Let $I=\mathbb{Z}$, $\CyclSpace(i\in I,\goth M_i,\phi_i)=\struct{X,H,\inc}$, and 
$$\alpha_i:=\phi_{0}^{-1}\ldots \phi_{i-2}^{-1}\phi_{i-1}^{-1},\;\;\;
\beta_i:=\phi_{-1}\ldots \phi_{i+1}\phi_{i}.$$
We introduce the following bijective map 
$\delta\colon \CyclSpace(i\in I,\goth M_i,\phi_i)
\longrightarrow \CyclSpacex(\mathbb{Z},\circ,{\goth M}_0)$ 
by the condition:
$$\delta((i,x))=\left\{ \begin{array}{ll}
               (i,\alpha_i(x))
               & \text{ for } i>0\\
               (i,x)  & \text{ for } i=0\\
               (i,\beta_i(x)) & \text{ for } i<0
              \end{array} \right.,\;\;
            (i,x)\in X\cup H.  $$
Let $(i,a)\inc [j,m]$ for some $(i,a)\in X$, $[j,m]\in H$.
First, we assume that $i=j$ and $a\inc_i m$. Then  
$\alpha_i(a)\inc_0 \alpha_i(m)$ (or $\alpha_i(m)\inc_0 \alpha_i(a)$) and
$\beta_i(a)\inc_0 \beta_i(m)$ (or $\beta_i(m)\inc_0 \beta_i(a)$) is evident.
\newline
Now let $i=j+1$ and $a=\phi''_j(m)=\phi''_{i-1}(m)$ and $i>0$. 
If $j\neq 0$ then we obtain
$\delta((i,a))=(i,\alpha_i(a))=(i,\alpha_{i-1}\phi_{i-1}^{-1}(a))=
(i,\alpha_{i-1}(m))$, and
$\delta((j,m))=\delta((i-1,m))=(i-1,\alpha_{i-1}(m))$. 
If $j=0$ then $a=\phi''_0(m)$ and we get
$\delta((i,a))=\delta((1,a))=(1,\phi_{0}^{-1}(a))=(1,m)$,
$\delta((j,m))=\delta((0,m))=(0,m)$.
We obtain the same result for $i\in\mathbb{Z}$ such that $i<0$.
If $i=0$ then $a=\phi''_{-1}(m)$, and $\delta((i,a))=(0,a)$, $\delta((j,m))=\delta((-1,\phi''_{-1}(m)))=(-1,a)$.
\newline
Each case considered above yields $\delta((i,a))\inc \delta((j,m))$, follows by \eqref{wz:corelinc}.
Consequently, the map $\delta$ is a required isomorphism.
\par
Now let $I=C_k$ with $k$ even. Consider the map
$$\sigma((i,x))=\left\{ \begin{array}{ll}
               (i,\alpha_i(x))
               & \text{ for } i\neq 0\\
               (i,x)  & \text{ for } i=0
              \end{array} \right.,\;\;
            (i,x)\in X\cup H.  $$
Obviously this map is a bijection acting from
$\CyclSpace(i\in I,\goth M_i,\phi_i)$ onto
$\CyclSpacex(C_k,\circ,{\goth M}_0)$.
Let us check whether $\sigma$ preserves the relation of incidence.
The only not evident (not analogous to the case with $I=\mathbb{Z}$) case we
have to consider is $(0,a)\inc [k-1,m]$.
Then $a=\phi''_{k-1}(m)$, $\sigma((0,a))=(0,a)$ and $\sigma((k-1,m))=(k-1,\alpha_{k-1}(m))$.
Assumption $\phi_{k-1}\ldots \phi_1\phi_0=\id$ 
yields $\phi_{k-1}=\phi^{-1}_0\ldots\phi^{-1}_{k-2}=\alpha_{k-1}$.
Therefore, $\sigma((k-1,m))=(k-1,\phi''_{k-1}(m))=(k-1,a)$ and $\sigma((0,a))\inc \sigma((k-1,m))$ holds
by \eqref{wz:corelinc}.
\end{proof}
\begin{prop}
Let $\CyclSpace(i\in I,\goth M_i,\phi_i)$ be the structure defined by
\eqref{def:glue}. If $I=C_k$ with odd $k$ and
$\kor=\phi_{k-1}\ldots \phi_1\phi_0$ is
an involutive correlation of ${\goth M}_0=\struct{\mathrm{S}_0,\mathrm{L}_0,\inc_0}$
then $\CyclSpace(i\in I,\goth M_i,\phi_i)\cong \CyclSpacex(k,\kor,{\goth M_0})$.
\end{prop}
\begin{proof}
Let $\CyclSpace(i\in I,\goth M_i,\phi_i)=\struct{X,H,\inc}$ and 
$$\alpha_i:=\phi_{0}^{-1}\ldots \phi_{i-2}^{-1}\phi_{i-1}^{-1},\;\;\;
\beta_i:=\phi_{k-1}\ldots \phi_{i+1}\phi_{i}.$$
We set the following bijective map $\delta$:
$$\delta((i,x))=\left\{ \begin{array}{ll}
               (i,\alpha_i(x))
               & \text{ for } i=2t\\
               (i,\beta_i(x)) & \text{ for } i=2t+1
              \end{array} \right.,\;\;
            (i,x)\in X\cup H.  $$
Let $(i,a)\inc [j,m]$ for some $(i,a)\in X$, $[j,m]\in H$. We claim that 
$\delta((i,a))\inc \delta((j,m))$ in sense of the relation $\inc$ introduced by
\eqref{wz:corelinc}.
The only not evident case is that with
$i=j+1$ and $a=\phi''_j(m)=\phi''_{i-1}(m)$. Let $i$ be even. Then we get
$\delta((i,a))=(i,\alpha_i(a))=
(i,\alpha_{i-1}\phi_{i-1}^{-1}(a))=
(i,\alpha_{i-1}(m))$, and
$\delta((j,m))=\delta((i-1,m))=(i-1,\beta_{i-1}(m))$.
Let $\kor_0=\phi_{k-1}\ldots \phi_1\phi_0$ be
an involutive correlation of ${\goth M}_0$.
Note, that formulas
$\kor(\beta_{i-1}(m))= \alpha_{i-1}(m)$ and $\kor^2(m)=m$ are equivalent,
and they hold since $\kor^2=id$.
In order to get the claim for odd $i$ we apply similar reasoning.
\par
Hence, the map $\delta$ is an isomorphism between the structures
$\CyclSpace(i\in I,\goth M_i,\phi_i)$ and $\CyclSpacex(k,\kor,{\goth M_0})$.
\end{proof}
%

\end{document}